\documentclass{tran-l}

\usepackage{bbm} 
\usepackage{xcolor}

\newtheorem{theorem}{Theorem}[section]

\theoremstyle{definition}

\newtheorem{example}[theorem]{Example}

\DeclareMathOperator{\sinc}{sinc}
\DeclareMathOperator{\Res}{Res}

\theoremstyle{remark}
\newtheorem{remark}[theorem]{Remark}

\numberwithin{equation}{section}

\begin{document}


\title[An approach to Borwein integrals from residue theory]{An approach to Borwein integrals \\ from the point of view of residue theory}
\author[D. Cao Labora]{Daniel Cao Labora$^1$} \footnote{ \tiny Daniel Cao Labora: Departament of Statistics, Mathematical Analysis and Optimization, Faculty of Mathematics and CITMAga, Universidade de Santiago de Compostela (USC), Galicia, Spain; daniel.cao@usc.es. ORCID: 0000-0003-2266-2075}
\address{}
\thanks{}

\author[G. Cao Labora]{Gonzalo Cao Labora$^2$}\footnote{ \tiny Gonzalo Cao Labora: Department of Mathematics, Massachusetts Institute of Technology (MIT), MA, United States of America; gcaol@mit.edu. ORCID: 0000-0002-8426-8391}



\date{May, 2024}


\begin{abstract}
Borwein integrals are one of the most popularly known phenomena in contemporary mathematics. They were found in 2001 by David Borwein and Jonathan Borwein and consist of a simple family of integrals involving the cardinal sine function ``sinc'', so that the first integrals are equal to $\pi$ until, suddenly, that pattern breaks. The classical explanation for this fact involves Fourier Analysis techniques. In this paper, we show that it is possible to derive an explanation for this result by means of undergraduate Complex Analysis tools; namely, residue theory. Besides, we show that this Complex Analysis scope allows to go a beyond the classical result when studying these kind of integrals. Concretely, we show a new generalization for the classical Borwein result.
\end{abstract}

\maketitle

\section{Introduction}

The integrals given by Borwein and Borwein in \cite{Borwein2001} involves the ``sinc'' function, that is defined as \[\textnormal{sinc}(x)=\frac{\sin x}{x},\] where one takes the definition $\textnormal{sinc}(0)=1$ in order to make the ``sinc'' function smooth on the whole real line $\mathbb{R}$. The surprising phenomenon involving the integration of products of rescaled versions of ``sinc'' is the following one. It can be seen that\footnote[3]{$I_1$ is not convergent in the Lebesgue sense but it is equal to $\pi$ in the Riemann sense. From $I_2$ onwards, all integrals can be interpreted either in the Lebesgue or Riemann sense.}
\begin{align*}
I_{1}:=&\int_{-\infty}^{\infty} \sinc(x) \, dx = \pi,\\
I_{2}:=&\int_{-\infty}^{\infty} \sinc(x) \sinc(x/3) \, dx = \pi,\\
I_{3}:=&\int_{-\infty}^{\infty} \sinc(x) \sinc(x/3) \sinc(x/5) \, dx = \pi, \\
I_{4}:=&\int_{-\infty}^{\infty} \sinc(x) \sinc(x/3) \sinc(x/5) \sinc(x/7) \, dx = \pi. \\
\end{align*}
so one would expect to have the property \[I_{n}:=\int_{-\infty}^{\infty} \prod_{j=1}^n \sinc\left(\frac{x}{2j-1}\right) dx = \pi\] for any number $n \in \mathbb{Z}^+$. Nevertheless, this reasonable conjecture fails for $n\geq 8$. 

The main objective of this paper is to prove the validity of the identity for $n \leq 7$ and to give an explanation for the failure whenever $n \geq 8$, by means of a direct use of Complex Analysis. The classical explanation for this phenomenon \cite{Borwein2001} involves Fourier Analysis, but, to the best of our knowledge, no full and intuitive explanation has been provided from the point of view of the theory of Complex Analysis. In Theorem \ref{th:three_dominant} we also provide a new result concerning Borwein integrals in the case where the first three frequencies are dominant (and not just the first one, as in the classical result). We are not aware of any proof of Theorem \ref{th:three_dominant} from the Fourier Analysis perspective, so this is also a demonstration of the power of the complex analytic approach.

Let us mention some of the work that has been done for Borwein integrals following \cite{Borwein2001}. In \cite{Borwein2002}, the authors consider similar ``sinc'' integrals in $\mathbb R^n$ and, using similar Fourier Analysis techniques, show a similar pattern in the multidimensional setting. There has also been work showing that this pattern persists when the integrals are replaced with summations \cite{Baillie2008}. In \cite[Theorem 1]{Bradley2002}, the authors arrive to the intermediate formula \eqref{eq:In2} (in an equivalent formulation) using principal value integrals, which is somewhat similar to the complex analytic one. However, the explicit formula $I_n = \pi$ for $n \leq 7$ is not obtained. Finally, let us mention that Borwein integrals have applications in the computation of the volume of intersection of hypercubes with Euclidean half-spaces \cite{Marichal2008} and in bounding the quantity of the integer solutions to linear equations \cite{Aaronson2019}.

We organize the paper as follows. In Section 2, we derive the classical result $I_n = \pi$ for $n\leq 7$ and provide an explanation for $I_n < \pi$ for $n \geq 8$, using basic techniques from Complex Analysis. In Section 3 we use the complex analytic approach to derive two generalizations. The first one, Theorem \ref{th:generic_freq}, is the case of arbitrary frequencies. This was already known from \cite{Borwein2001}, but we provide the first proof from Complex Analysis. A further generalization, which is new, is given in Theorem \ref{th:three_dominant} and corresponds to the case where the first frequency is not dominant, but the first three frequencies are dominant.

For the rest of this paper, we will deal with the natural extension of the ``sinc'' function to $\mathbb{C}$. This entire function is written in terms of the exponential function as \[\textnormal{sinc}(z)=\frac{e^{iz}-e^{-iz}}{2iz},\] where we have simply used the well known expression for the complex sine function.

\section{A Complex Analysis explanation for the original result}

In this part of the manuscript we will provide a simple explanation for the phenomenon involving Borwein integrals that was described in the previous section. We will only require a basic knowledge of usual tools from an undergraduate course in Complex Variable; namely, elementary results involving residue theory. 

\subsection{General Strategy from Complex Analysis}

For the rest of the paper, $g(z)$ will denote an entire function on the complex plane whose restriction to the real line is integrable. The function $g(z)$ will have the special property that it can be split as the sum of two holomorphic functions on $\mathbb{C} \setminus \{0\}$, namely $g(z) = g_1(z) + g_2(z)$, each of them having a unique pole of order $n$ that will be located at $z=0$. Besides, the integral of $g_1(z)$ along the semicircumference of center $z=0$ and radius $R$ on the upper half plane will tend to zero when $R \to \infty$. The same will happen for $g_2(z)$ on the lower half plane. In the first subsection, we will show how the residue of $g_1(z)$ at $0$ denoted by $\Res(g_1,0)$ is enough to determine the value of $\int_{-\infty}^{\infty} g(x) \, dx.$

In the second subsection, we will consider the particular case where \[g(z)= \prod_{j=1}^n \sinc \left(\frac{z}{2j-1}\right).\] Let us argue that this $g(z)$ admits the $g_1(z) + g_2(z)$ decomposition described above. The complex expression for ``sinc'' and the expansion of the product of the $n$ factors will produce $2^n$ summands of the form \[ \pm \frac{(2n-1)!!}{(2i)^n} \, \frac{ e^{\lambda i z}}{z^n},\] for some values of $\lambda$ and some choices of $+$ or $-$ in $\pm$. The function $g_1$ will be the sum of the terms where $\lambda>0$, whether $g_2$ will be the sum of the terms where $\lambda<0$. Hence, the second subsection will be devoted to the computation of the quantity $\Res(g_1,0)$, that will be expressed as the sum of $2^{n-1}$ elements, since half of the $2^n$ terms have $\lambda > 0$. 

In the third subsection, after a combinatorial argument, we will simplify the sum of $2^{n-1}$ terms when $n \leq 7$ and we will prove that, in such a case, $\int_{-\infty}^{\infty} g(x) \, dx = \pi.$ Finally, in the fourth subsection, we will provide an expression for the difference between $\int_{-\infty}^{\infty} g(x) \, dx$ and $\pi$ when $n = 8$.

\subsection{Calculation of Borwein integrals in terms of residues} \label{Calculation}

First, we show how does $\Res(g_1,0)$ determine the value of the integral of $g$ along the real line. To do this, we will consider the following paths, where $a \in \mathbb{R}^+$ and $t \in [0,\pi]$: \[\mu_{\uparrow,a}(t)=a e^{it}, \quad \mu_{\downarrow,a}(t)=-a e^{it}.\] Thus, both families of paths consist on positively oriented semicircumferences. The reason for such a notation is that for $\mu_{\uparrow,a}$ the arc goes from $a$ to $-a$ by the upper half plane, and for $\mu_{\downarrow,a}$ the arc goes from $-a$ to $a$ by the lower half plane. Using some oriented segments, and the previous oriented arcs, we define the following closed contours for latter integration along them
\begin{align*}
\gamma_{\uparrow, R, \varepsilon}&:=[-R,-\varepsilon] \cup \mu_{\downarrow,\varepsilon} \cup [\varepsilon, R] \cup \mu_{\uparrow, R},\\
\gamma_{\downarrow, R, \varepsilon}&:=[-R,-\varepsilon] \cup \mu_{\downarrow,\varepsilon} \cup [\varepsilon, R] \cup \mu_{\downarrow, R}.
\end{align*}
Besides, the decay properties on $g_1$ and $g_2$ suggest us to consider the following integrals along the previous contours
\begin{align*}
\int_{\gamma_{\uparrow, R, \varepsilon}} g_1(z) \, dz =\int_{[-R,-\varepsilon]} g_1(z) \, dz  +\int_{\mu_{\downarrow,\varepsilon}} g_1(z) \, dz  + \int_{[\varepsilon, R]} g_1(z) \, dz  + \int_{\mu_{\uparrow, R}} g_1(z) \, dz \\
\int_{\gamma_{\downarrow, R, \varepsilon}} g_2(z) \, dz =\int_{[-R,-\varepsilon]} g_2(z) \, dz  +\int_{\mu_{\downarrow,\varepsilon}} g_2(z) \, dz  + \int_{[\varepsilon, R]} g_2(z) \, dz  + \int_{\mu_{\downarrow, R}} g_2(z) \, dz
\end{align*}
Thanks to residue theory, taking into account the indexes of $\gamma_{\uparrow, R, \varepsilon}$ and $\gamma_{\downarrow, R, \varepsilon}$ with respect to the unique pole at $0$, we can write the exact value of the integrals along the closed contour
\begin{align*}
\int_{[-R,-\varepsilon]} g_1(z) \, dz  +\int_{\mu_{\downarrow,\varepsilon}} g_1(z) \, dz  + \int_{[\varepsilon, R]} g_1(z) \, dz  + \int_{\mu_{\uparrow, R}} g_1(z) \, dz &= 2 \pi i \Res(g_1,0), \\
\int_{[-R,-\varepsilon]} g_2(z) \, dz  +\int_{\mu_{\downarrow,\varepsilon}} g_2(z) \, dz  + \int_{[\varepsilon, R]} g_2(z) \, dz  + \int_{\mu_{\uparrow, R}} g_2(z) \, dz &= 0.
\end{align*}
Now, the addition of the two previous expressions together with the the linearity of the integral gives us
\[\int_{[-R,-\varepsilon]} g(z) \, dz  +\int_{\mu_{\downarrow,\varepsilon}} g(z) \, dz  + \int_{[\varepsilon, R]} g(z) \, dz  + \int_{\mu_{\uparrow, R}} g_1(z) \, dz + \int_{\mu_{\downarrow, R}} g_2(z) \, dz = 2 \pi i \Res(g_1,0).\]
Finally, we consider the double limit when $\varepsilon \to 0$ and $R \to \infty$ and analyse what happens to each addend.

First, by definition and since $g$ is integrable, \[\lim_{\substack{R \to \infty \\ \varepsilon \to 0}} \left( \int_{[-R,-\varepsilon]} g(z) \, dz + \int_{[\varepsilon, R]} g(z) \, dz  \right)= \int_{-\infty}^{\infty} g(z) \, dz.\]

Second, since $g$ is entire and the length of $\mu_{\downarrow,\varepsilon}$ tends to zero when $\varepsilon \to 0$, \[\lim_{\varepsilon \to 0} \int_{\mu_{\downarrow,\varepsilon}} g(z) \, dz  = 0.\]

Third, the decay of the integrals of $g_1$ and $g_2$ on the corresponding half plane guarantees \[\lim_{R \to \infty}\int_{\mu_{\uparrow, R}} g_1(z) \, dz= \lim_{R \to \infty} \int_{\mu_{\downarrow, R}} g_2(z) \, dz=0.\]

The combination of all this information provides the formula \[\int_{-\infty}^{\infty} g(z) \, dz = 2 \pi i \Res(g_1,0).\]

\subsection{Calculation of the residue} \label{Residue}

As we have sketched before, up to some minus signs, $g_1(z)$ is the sum of $2^{n-1}$ elements of the form \[\frac{(2n-1)!!}{(2i)^n} \, \frac{ e^{\lambda i z}}{z^n}\] for certain values of $\lambda$. Consequently, it will be convenient to calculate the value of the residue of the following pole at $z=0$: 
\begin{equation} \label{eq:Res}
\textnormal{Res}\left(\frac{(2n-1)!!}{(2i)^n} \, \frac{ e^{\lambda i z}}{z^n},0\right).
\end{equation}

Since the order of the pole is $n$, we can compute \[\textnormal{Res}\left(\frac{(2n-1)!!}{(2i)^n} \, \frac{ e^{\lambda i z}}{z^n},0\right)=\frac{1}{(n-1)!} \frac{(2n-1)!!}{(2i)^n} \, {\frac{d^{n-1}}{dz^{n-1}} e^{\lambda iz}}_{\big \vert z=0}.\]

Finally, immediate calculations show that \[ \textnormal{Res}\left(\frac{(2n-1)!!}{(2i)^n} \, \frac{ e^{\lambda i z}}{z^n},0\right) = \frac{1}{2i} \frac{(2n-1)!!}{(n-1)!} \left(\frac{\lambda}{2}\right)^{n-1}.\] 

\subsection{Explicit expression for the residue for $n \leq 7$} \label{Explicit}

According to the two previous subsections, if we consider \[\lambda_{\sigma}=\sigma_1 + \sigma_2 \cdot \frac{1}{3} + \sigma_{3} \cdot \frac{1}{5} + \dots + \sigma_{n} \cdot \frac{1}{2n-1},\] where $\sigma=(\sigma_1,\dots,\sigma_n) \in \{-1,1\}^n$, we have that 
\begin{equation} \label{eq:In1}
I_n= 2\pi i \hspace{-0.4 cm} \sum_{\substack{\lambda_{\sigma}>0 \\ \sigma \in \{-1,1\}^n}} \hspace{-0.4 cm} \Res\left(\frac{(2n-1)!!}{(2i)^n} \, \frac{ e^{\lambda_{\sigma} i z}}{z^n} \, \prod_{j=1}^n \sigma_j ,0\right)=\pi\hspace{-0.4 cm}\sum_{\substack{\lambda_{\sigma}>0 \\ \sigma \in \{-1,1\}^n}} \hspace{-0.4 cm} \frac{(2n-1)!!}{(n-1)!} \left(\frac{\lambda_{\sigma}}{2}\right)^{n-1} \prod_{j=1}^n \sigma_j ,
\end{equation}
for any $n \in \mathbb{Z}^+$. In this section, we will prove that the latter quantity is exactly $\pi$, whenever $n \leq 7$.

In order to do so, the key observation is that the condition $\sigma_1=1$ is equivalent to $\lambda_{\sigma}>0$ if $n \leq 7$, but not in general. The reason is that have
\begin{equation*}
\left \vert \sum_{j=2}^n \sigma_{j} \cdot \frac{1}{2j-1} \right\vert < \frac{1}{3} + \frac{1}{5} + \ldots + \frac{1}{2n-1} < 1,
\end{equation*}
given that $n \leq 7$. Therefore, the assumption $n \leq 7$ implies $\lambda_\sigma > 0$ whenever $\sigma_1 = 1$ and $\lambda_\sigma < 0$ whenever $\sigma_1 = -1$. 
Then, because of \eqref{eq:In1}, we can express $I_n$ as
\begin{align} 
I_n= \frac{\pi}{2^{n-1}} \frac{(2n-1)!!}{(n-1)!}  \underbrace{ \sum_{\substack{ \sigma_1=1 \\ \sigma \in \{-1,1\}^{n}}  }
 \left( \sum_{j=1}^{n} \frac{\sigma_j}{2j-1} \right)^{n-1} \prod_{j=1}^{n} \sigma_j }_{\mathcal S} \label{eq:In2}
\end{align}

Now, we proceed to study the sum $\mathcal S$. Let us denote with the symbol $\mathcal P$ the family of $n$-tuples of nonnegative integers $(p_1, p_2, \ldots , p_n) \in \mathbb Z^{n}_{\geq 0}$ fulfilling the condition $p_1 + p_2 + \ldots + p_{n-1}+ p_n= n-1$. The multinomial formula allows us to express
\begin{equation} \label{eq:S}
\mathcal S =   \sum_{\substack{ \sigma_1=1 \\ \sigma \in \{-1,1\}^{n}}  }
 \left(\sum_{p \in \mathcal P} \binom{n-1}{p} \prod_{j=1}^{n} \left( \frac{\sigma_j}{2j-1} \right)^{p_j} \right) \prod_{j=1}^{n} \sigma_j
\end{equation}
where we recall the definition for multinomial coefficients
\begin{equation*}
\binom{n-1}{p} = \frac{(n-1)!}{p_1! \, p_2! \cdots  p_n!}.
\end{equation*}
If we group the $\sigma_{j}$ factors, and after taking into account that $\sigma_1=1$, we can rewrite \eqref{eq:S} as
\begin{equation} \label{eq:S2}
\mathcal S = \sum_{p \in \mathcal P} \left(\binom{n-1}{p} \left(\prod_{j=1}^{n} \left( \frac{1}{2j-1} \right)^{p_j} \right)  \left( \sum_{\substack{ \sigma_1=1 \\ \sigma \in \{-1,1\}^{n}}  } \prod_{j=2}^{n} \sigma_j^{p_j+1}\right)\right).
\end{equation}
For each fixed $p \in \mathcal{P}$, the sum over $(\sigma_2,\dots,\sigma_{n}) \in \{-1,1\}^{n-1}$ appearing in \eqref{eq:S2} can be factorised, yielding
\begin{equation} \label{eq:symmetrization_equality}
\sum_{\substack{ \sigma_1=1 \\ \sigma \in \{-1,1\}^{n}}  } \prod_{j=2}^{n} \sigma_j^{p_j+1} =
  \prod_{j=2}^{n} \left( \sum_{\sigma_j \in \{-1,1\}}   \sigma_j^{p_j+1} \right)
  = \prod_{j=2}^{n} 2 \cdot \mathbbm{1}_{p_j \text{ is odd}}.
\end{equation}
Since $p \in \mathcal P$, we have that $p_1 + p_2 + \ldots + p_{n} = n-1$, so the only possible situation where every $p_j$ is odd for $j \geq 2$ consists on the case where $p_j = 1$ for every $j \in \{2, 3, \ldots n \}$. Hence, this claim allows to simplify \eqref{eq:S2} by using \eqref{eq:symmetrization_equality}, obtaining
\begin{equation*}
\mathcal S = (n-1)!  \left( \prod_{j=1}^{n} \frac{1}{2j-1} \right) 2^{n-1} = \frac{ 2^{n-1} (n-1)!}{(2n-1)!!}
\end{equation*}
Finally, after plugging this value for $\mathcal{S}$ into \eqref{eq:In2} and making direct cancellations, we conclude 
\begin{equation*}
I_n =  \pi
\end{equation*}
whenever $n \leq 7$.
\subsection{The residue for $n \geq 8$}
The three previous subsections provide a proof for the claim $I_n= \pi$ for any $n \leq 7$. Indeed, if one checks the arguments that have been made, no special mention to the $n \leq 7$ case has been done in subsections \ref{Calculation} or \ref{Residue}. Nevertheless, as we stated in subsection \ref{Explicit}, the key fact is that the sign of the expression \[\sum_{j=1}^n \sigma_j \cdot \frac{1}{2j-1}\] is governed by the sign of $\sigma_1$ whenever $n \leq 7$. However, this is no longer true for $n \geq 8$. If we recall \eqref{eq:In1}, we have the following expression for $I_n$
\begin{equation*}\label{I8}
I_n=\pi \hspace{-0.4 cm}\sum_{\substack{\lambda_{\sigma}>0 \\ \sigma \in \{-1,1\}^n}} \hspace{-0.4 cm} \frac{(2n-1)!!}{(n-1)!} \left(\frac{\lambda_{\sigma}}{2}\right)^{n-1} \prod_{j=1}^n \sigma_j .
\end{equation*}
Besides, due to the previous subsection, we know that
\begin{equation}\label{pi8}
\pi=\pi \hspace{-0.4 cm}\sum_{\substack{\ \sigma_1=1 \\ \sigma \in \{-1,1\}^n}} \hspace{-0.4 cm} \frac{(2n-1)!!}{(n-1)!} \left(\frac{\lambda_{\sigma}}{2}\right)^{n-1} \prod_{j=1}^n \sigma_j.
\end{equation}
In particular, if $n=8$, the difference between \eqref{I8} and \eqref{pi8} comes from the number \[\lambda^*=1-\frac{1}{3}-\frac{1}{5}-\frac{1}{7}-\frac{1}{9}-\frac{1}{11}-\frac{1}{13}-\frac{1}{15}.\] Observe that $\lambda^*<0$, although the first addend $1$ has positive sign, and that there is an odd amount of minus signs in $\lambda^*$ and in $-\lambda^*$. So, in order to get the correct result for $I_8$, departing from the expression for $\pi$ in \eqref{pi8}, one has to quit the contribution of $\lambda^*$ and add the contribution of $-\lambda^*$. Consequently, \[I_8=\pi \left(1-(-1)^{7}\frac{(2 \cdot 8-1)!!}{7!} \left(\frac{\lambda^*}{2}\right)^{8-1}+(-1)^7\frac{(2 \cdot 8-1)!!}{7!} \left(-\frac{\lambda^*}{2}\right)^{8-1}\right).\] A tedious, but standard, sequence of operations produces \[I_8=\pi \left( 1- \frac{1}{2^6} \cdot \frac{15!!}{7!} \cdot \vert \lambda^* \vert^7 \right) =\pi \left( 1- \frac{6\,879\,714\,958\,723\,010\,531}{467\,807\,924\,720\,320\,453\,655\,260\,875\,000}\right).\] Of course, it would be possible to calculate the exact value for $I_9,I_{10},\dots$ introducing the pertinent corrections in the previous expression. These corrections would imply considering corrections relative to values $\lambda_{\sigma}$ such that $\lambda_{\sigma}<0$ despite having $\sigma_1=1$. However, this procedure would become longer as $n$ gets larger.

\section{Some generalizations}

In this section we explore some extensions of the classical Borwein result developed in the previous section. On the one hand, it is easy to guess that the specific sequence of values for the frequencies $1,1/3,1/5,1/7\dots$ is not relevant, but the facts \[\frac{1}{3}+\frac{1}{5}+\cdots+\frac{1}{13}<1\] and \[\frac{1}{3}+\frac{1}{5}+\cdots+\frac{1}{15}>1\] are the keys for explaining the break of the pattern. In other words, if each frequency is associated to a plus or minus sign, the important fact is that the sign for the first frequency determines the sign of the sum of all frequencies. On the other hand, and again from the point of view of plus and minus signs, we could ask: what happens if the sign for the sum of the frequencies is determined by the signs of the three first frequencies? This latter question is, to the best of our knowledge, unanswered and we use the previously developed techniques in order to provide a response to it.

\subsection{Arbitrary frequencies in Borwein integrals}

\begin{theorem} \label{th:generic_freq} Consider a non-increasing sequence of positive real numbers $a_j \in \mathbb{R}^+$ for every $j \in \mathbb{Z}^+$, such that exists $N \in \mathbb{N}_{\geq 2}$ with \[a_1 > \sum_{j=2}^N a_j, \textnormal{ but } a_1 < \sum_{j=2}^{N+1} a_j.\] Then, for any $n \leq N$, we have that \[I_n:= \int_{-\infty}^{\infty} \prod_{j=1}^n \sinc\left(a_j \, x\right) \, dx = \frac{\pi}{a_1}, \textnormal{ but } I_{N+1}:= \int_{-\infty}^{\infty} \prod_{j=1}^{N+1} \sinc\left(a_j \, x\right) \, dx \neq \frac{\pi}{a_1}.\]
\end{theorem}

\begin{proof}
The proof for the case $n \leq N$ is trivial, after taking into account the considerations in the previous Section. If we establish the change of variables $y=a_1\,x$, we have that \[I_n=\frac{1}{a_1}\int_{-\infty}^{\infty} \prod_{j=1}^n \sinc\left(\frac{a_j}{a_1} \, y\right) \, dy.\] Now that the first frequency equals $1$, we can apply the same argument of the previous Section, just replacing the roles of the particular values $1/(2j-1)$ with $a_j/a_1$ for every $j \geq 2$. Note that these specific values are not important. Before we had a cancellation of $(2n-1)!!$ on a numerator and on a denominator, and now we would cancel the factor $\prod_{j=1}^n a_1/a_j$. Consequently, \[I_n=\frac{1}{a_1}\int_{-\infty}^{\infty} \prod_{j=1}^n \sinc\left(\frac{a_j}{a_1} \, y\right) \, dy = \frac{\pi}{a_1},\] whenever $n \leq N$.

For the case $N+1$, the difference between $I_{N+1}$ and $\pi/a_1$ comes from the number \[\lambda^*=1-\frac{a_2}{a_1}-\frac{a_3}{a_1}-\ldots-\frac{a_{N+1}}{a_1}.\] Observe that $\lambda^*<0$, although the first addend $1$ has positive sign, and that there is an amount of $N$ minus signs in $\lambda^*$ and in $-\lambda^*$. So, in order to get the correct result for $I_{N+1}$, departing from $\pi/a_1$, one has to quit the contribution of $\lambda^*$ and add the contribution of $-\lambda^*$. Consequently, \[I_{N+1}=\frac{\pi}{a_1} \left(1-(-1)^N\frac{\prod_{j=1}^{N+1} (a_1/a_j)}{N!} \left(\frac{\lambda^*}{2}\right)^{N}+(-1)^1\frac{\prod_{j=1}^{N+1} (a_1/a_j)}{N!} \left(-\frac{\lambda^*}{2}\right)^{N}\right).\] From this expression, we deduce \[I_{N+1}=\frac{\pi}{a_1} \left( 1-\frac{1}{2^{N-1}} \cdot \frac{\prod_{j=1}^{N+1} (a_1/a_j)}{N!} \cdot \vert \lambda^* \vert^N \right).\] Since the number $a_1 \lambda^*$ is easier to be computed than $\lambda^*$, we also give the formula \[I_{N+1}=\pi \left( \frac{1}{a_1}-\frac{1}{2^{N-1}} \cdot \frac{1}{N! \prod_{j=1}^{N+1} a_j} \cdot \vert a_1\, \lambda^* \vert^N \right).\]
\end{proof}

\subsection{The case where the three first frequencies are dominant}

The complex analytic techniques developed in this article allow us to compute new Borwein integrals in cases where the sign of $\lambda_\sigma$ is not determined by $\sigma_1$, but it is determined by $\{ \sigma_1, \sigma_2, \sigma_3\}$. To our best knowledge, the computation of these integrals is new. The complex analytic approach presented in this paper allows to compute those integrals, whereas to our best knowledge the standard Fourier Transform approach does not allow such a treatment. 

\begin{theorem} \label{th:three_dominant} Let $n \in \mathbb N_{\geq 3}$ and consider a finite non-increasing sequence of positive real numbers $a_j \in \mathbb{R}^+$, where $1 \leq j \leq n$, such that
\begin{equation} \label{eq:condition}
a_2 + a_3 - a_1 > \sum_{k=4}^n a_k.
\end{equation}
Then, we have that 
\[I_n:= \int_{-\infty}^{\infty} \prod_{j=1}^n \sinc\left(a_j \, x\right) \, dx = 
\pi \cdot  \frac{- \sum_{k=1}^n a_k^2 - 2(a_1^2 + a_2^2 + a_3^2) + 6(a_1a_2 + a_2a_3 + a_1a_3)}{12 a_1 a_2 a_3}.\]
\end{theorem}

Note that the condition of the theorem implies $-a_1 + a_2 + a_3 > 0$, so, in particular, we know that we are not in the hypothesis of the first theorem. 

\begin{proof} We use the same notation as before for $\lambda_\sigma = \sum_{j=1}^n \sigma_j \lambda_j$. We know from the previous section that
\begin{equation} \label{eq:malaysia}
I_n = \frac{\pi}{2^{n-1}} \frac{1}{(n-1)!}  \frac{1}{\prod_{j=1}^n a_j} 
\underbrace{ \sum_{ \substack{ \lambda_\sigma > 0 \\ \sigma \in \{-1, 1\}^n } } \lambda_\sigma^{n-1} \prod_{j=1}^n \sigma_j }_{\mathcal S}.
\end{equation}

First, let us show that
\begin{equation} \label{eq:signs}
\text{sign}(\lambda_\sigma) = \text{sign} \left( \sigma_1 + \sigma_2 + \sigma_3 \right),
\end{equation} 
that is, $\lambda_\sigma > 0$ if and only if at least two of the first three signs are positive. Indeed, the condition \eqref{eq:condition} gives us that if $\sigma_1 = -1$, $\sigma_2 = \sigma_3 = +1$, then $\lambda_\sigma > 0$. The fact that $a_j$ is positive and non-increasing implies
\begin{equation*}
a_1 + a_2 + a_3, a_1 - a_2 + a_3, a_1 + a_2 - a_3 \geq -a_1 + a_2 + a_3
\end{equation*}
so any other arrangement with at least two positive signs among $\sigma_1, \sigma_2, \sigma_3$ yields $\lambda_\sigma > 0$ as well. An analogous argument gives that if at least two signs among $\sigma_1, \sigma_2, \sigma_3$ are negative, then $\lambda_\sigma < 0$. We conclude \eqref{eq:signs}.

Using \eqref{eq:signs} and the inclusion-exclusion principle, we can decompose $\mathcal S$ in \eqref{eq:malaysia} as
\begin{align}
\mathcal S &= \mathcal S_{\mathrm{12}} + \mathcal S_{\mathrm{13}} + \mathcal S_{\mathrm{23}} - 2\mathcal S_{\mathrm{123}}, \qquad \mbox{ where }  \label{eq:Sdec} \\
\mathcal S_{\mathrm{ij}} &=  \sum_{ \substack{ \sigma_i, \sigma_j = 1 \\ \sigma \in \{-1, 1\}^n } } \lambda_\sigma^{n-1} \prod_{j=1}^n \sigma_j \qquad \mbox{ and } \qquad 
\mathcal S_{\mathrm{123}} =  \sum_{ \substack{ \sigma_1=\sigma_2=\sigma_3 = 1 \\ \sigma \in \{-1, 1\}^n } } \lambda_\sigma^{n-1} \prod_{j=1}^n \sigma_j. \notag
\end{align}

The computations of $\mathcal S_{\mathrm{ij}}$ and $\mathcal S_{\mathrm{123}}$ is similar to the one done in Section \ref{Explicit}, so we will just focus on the parts that are different. We start computing $\mathcal S_{\mathrm{12}}$. Recall that $\mathcal P$ is the family formed by tuples $p = (p_1, p_2, \ldots , p_n)$ with $p_i \in \mathbb N_{\geq 0}$ and $p_1 + p_2 + \ldots + p_n = n-1$. Using the multinomial formula for $\lambda_\sigma^{n-1} = \left( \sum_{j=1}^n \sigma_j a_j \right)^{n-1}$, we have
\begin{align}  
\mathcal S_{\mathrm{12}} &= \sum_{p \in \mathcal P} \binom{n-1}{p} \prod_{j=1}^n a_j^{p_j} \sum_{ \substack{ \sigma_1 = \sigma_2 = 1 \\ \sigma \in \{-1, 1\}^n } }  \prod_{j=1}^n \sigma_j^{p_j+1} \notag\\
&= \sum_{p \in \mathcal P} \binom{n-1}{p} \prod_{j=1}^n a_j^{p_j} \prod_{j=3}^n \left( (-1)^{p_j+1} + 1^{p_j+1} \right)
\label{eq:indonesia}
\end{align}
Note that the last sum is zero unless $p_3, p_4, \ldots p_n$ are all odd. Given that $p_3 + p_4 + \ldots p_n \leq n-1$, they must be all $1$. Thus, the only possible values of $p \in \mathcal P$ that yield a non-zero contribution in the above sum are $p = (0, 1, 1, 1, 1, \ldots , 1)$ and $p = (1, 0, 1, 1, 1, \ldots , 1)$. In those cases, $\prod_{j=1}^n \sigma_j^{p_j+1} = 1$. Thus, using this observation in \eqref{eq:indonesia}, we have
\begin{equation*}
\mathcal S_{\mathrm{12}} =  (n-1)!  (a_1 + a_2) \prod_{j=3}^n a_j \cdot 2^{n-2}  = 2^{n-1} (n-1)! \prod_{j=1}^n a_j \cdot \frac{1}{2} \left( \frac{1}{a_1} + \frac{1}{a_2} \right).
\end{equation*}

Reasoning in an analogous way, one can conclude that 
\begin{equation*}
\mathcal S_{\mathrm{13}} = 2^{n-1} (n-1)! \prod_{j=1}^n a_j \cdot \frac{1}{2} \left( \frac{1}{a_1} + \frac{1}{a_3} \right), \quad \mbox{ and } \quad 
\mathcal S_{\mathrm{23}} = 2^{n-1} (n-1)! \prod_{j=1}^n a_j \cdot \frac{1}{2} \left( \frac{1}{a_2} + \frac{1}{a_3} \right).
\end{equation*}
Thus, we obtain
\begin{equation} \label{eq:Sother}
\mathcal S_{\mathrm{12}} + \mathcal S_{\mathrm{13}} + \mathcal S_{\mathrm{23}}
= 2^{n-1} (n-1)! \prod_{j=1}^n a_j \cdot \left( \frac{1}{a_1} + \frac{1}{a_2} + \frac{1}{a_3} \right)
\end{equation}

Similarly to \eqref{eq:indonesia}, the multinomial formula of $\lambda_\sigma^{n-1}$ on $\mathcal S_{\mathrm{123}}$ yields
\begin{equation} \label{eq:timorleste}
\mathcal S_{\mathrm{123}} = \sum_{p \in \mathcal P} \binom{n-1}{p} \prod_{j=1}^n a_j^{p_j} \prod_{j=4}^n \left( (-1)^{p_j+1}+ 1^{p_j+1} \right) .
\end{equation}
The latter sum is zero unless all $p_j$ are odd for $j > 4$. Thus, $p_j \geq 1$ for $j \geq 1$. We have $p_1 + p_2 + p_3 + \sum_{j=4}^n (p_j - 1) = 2$, where all the summands are non-negative integers and moreover the terms $(p_j-1)$ for $j \geq 4$ are even. Thus, there are two disjoint possibilities on which the contribution from \eqref{eq:timorleste} is non-zero:
\begin{itemize}
\item $p_k = 3$ for some $k \geq 4$. Then, $p_1 = p_2 = p_3 = 0$ and $p_j = 1$ for all $j \geq 4$, $j \neq k$. We denote such set of tuples $p$ by $\mathcal P_A$
\item All $p_j = 1$ for $j \geq 4$. Then, $p_1 + p_2 + p_3 = 2$. We denote such set of tuples $p$ by $\mathcal P_B$.
\end{itemize}
We divide the sum on \eqref{eq:timorleste} as $\mathcal S_{\mathrm{123}} = \mathcal S_A + \mathcal S_B$, where $\mathcal S_A$ corresponds to the terms of the sum with $p \in \mathcal P_A$ and $\mathcal S_B$ to the terms with $p \in \mathcal P_B$. We have that
\begin{equation*}
\mathcal S_A = \sum_{k=4}^n \frac{(n-1)!}{3!} \prod_{j=4}^n a_j \cdot a_k^2 \cdot 2^{n-3}
= 2^{n-1} (n-1)! \prod_{j=1}^n a_j \cdot \frac{\sum_{k=4}^n a_k^2 }{24 a_1 a_2 a_3},
\end{equation*}
and 
\begin{align*}
\mathcal S_B &= \sum_{p_1 + p_2 + p_3 = 2} \frac{(n-1)!}{p_1!p_2!p_3!} \prod_{j=1}^n a_j \cdot a_1^{p_1-1} a_2^{p_2-1} a_3^{p_3-1}  \cdot 2^{n-3} \\
&=
2^{n-1} (n-1)! \prod_{j=1}^n a_j \left( \frac{1}{8} \cdot \frac{a_1^2+a_2^2+a_3^2}{a_1a_2a_3} + \frac{1}{4} \cdot \frac{a_1a_2 + a_1a_3 + a_2a_3}{a_1a_2a_3}\right) \\
&=
2^{n-1} (n-1)! \prod_{j=1}^n a_j \cdot \frac{(a_1+a_2+a_3)^2}{8a_1a_2a_3}
\end{align*}
Therefore, we get
\begin{align} 
\mathcal S_{\mathrm{123}} &= 2^{n-1} (n-1)! \prod_{j=1}^n a_j \left( 
\frac{\sum_{k=4}^n a_k^2 }{24 a_1 a_2 a_3} +  \frac{(a_1+a_2+a_3)^2}{8a_1a_2a_3} \right)
\label{eq:S123}
\end{align}

Now, substituting \eqref{eq:Sother} and \eqref{eq:S123} on \eqref{eq:Sdec}, we get
\begin{align*}
\mathcal S &= \mathcal S_{\mathrm{12}} + \mathcal S_{\mathrm{13}} + \mathcal S_{\mathrm{23}} - 2 \mathcal S_{\mathrm{123}}  \\
&=
 2^{n-1} (n-1)! \prod_{j=1}^n a_j\cdot\frac{- \sum_{k=1}^n a_k^2 - 2(a_1^2 + a_2^2 + a_3^2) + 6(a_1a_2 + a_2a_3 + a_1a_3)}{12a_1a_2a_3} .
\end{align*}

Substituting this into \eqref{eq:malaysia}, we get
\begin{equation}
I = \pi \cdot  \frac{- \sum_{k=1}^n a_k^2 - 2(a_1^2 + a_2^2 + a_3^2) + 6(a_1a_2 + a_2a_3 + a_1a_3)}{12 a_1 a_2 a_3}.
\end{equation}

\end{proof}

From the previous theorem, we can derive the following straightforward remarks and examples.

\begin{remark} In the case $a_1 = a_2$, the expression simplifies and we get 
\begin{equation}
I_n = \pi \cdot \left(\frac{1}{a_1} - \frac{a_3}{4a_1^2} - \frac{1}{12 a_3 a_1^2} \sum_{k=4}^n a_k^2 \right).
\end{equation}
\end{remark}

\begin{remark} The case $n=3$ has been used as an example in the literature due to the lack of closed general formulas. In the case $n=3$, our result simplifies to \[I_3=\frac{\pi}{a_1 a_2 a_3} \left(  \frac{2(a_1 a_2+ a_2 a_3+ a_3 a_1)-(a_1^2+a_2^2+a_3^2)}{4}\right),\] which agrees with the previous literature (for example, Equation (4.5) in \cite{Aaronson2019}).
\end{remark}

\begin{remark} It is also possible to consider the limit case of Theorem \ref{th:three_dominant} when $n \to \infty$ due to the dominated convergence theorem, since $\vert\sinc(a_j x)\vert \leq 1$.
\end{remark}

Contrary to the case when $a_1$ is dominant, the formula when the three biggest frequencies dominate involves all the coefficients $a_j$, concretely, its $\ell^2$ norm. 

\begin{example} Consider a finite non-increasing sequence of positive real numbers $a_j \in \mathbb{R}^+$ for $1 \leq j \leq n$, where $a_1=a_2=a_3=1$. If we assume that $\sum_{j=4}^n a_j < 1$, then we have that
\begin{equation*}
\int \prod_{j=1}^n \sinc (a_j x) dx = \pi \left( 1 - \frac{1}{12} \| a \|_{\ell^2}^2 \right)= \pi \left( 1 - \frac{1}{12} \sum_{j=1}^{n} a_j^2 \right).
\end{equation*}
\end{example}

\begin{example} Let us consider the sequence $a_j = \frac{1}{j!}$ for $j \geq 0$. It is well known that $\sum_{j=0}^\infty a_j = e$, so we have that $a_1 + a_2 - a_0 > \sum_{j=3}^\infty a_j$ since $1/2 > e-5/2$. Thus, we obtain the identity
\begin{equation*}
\int_{-\infty}^\infty \prod_{j=0}^\infty \sinc \left( \frac{x}{j!} \right) dx =\pi \left( \frac{5}{4} - \frac{1}{6}\| a \|_{\ell^2}^2\right) =\pi \left( \frac{5}{4} - \frac{1}{6} \sum_{j=0}^\infty \frac{1}{j!^2} \right).
\end{equation*}
\end{example}

\section*{Declarations}

\begin{itemize}
\item Availability of data and materials: Not applicable.
\item Competing interests: The authors declare that they do not have competing interests.
\item Funding: Funding for Daniel Cao Labora was provided by Xunta de Galicia (Grant No. ED431C 2019/02), Spanish National Plan for Scientific and Technical Research and Innovation (Grant No. MTM2016-75140-P).
\item Authors contributions: Both authors discussed together the ideas of the paper. Besides, both authors wrote and reviewed the manuscript together.
\item Acknowledgments: Not applicable.
\item Authors information:\\

Daniel Cao Labora \quad daniel.cao@usc.es \quad ORCID: 0000-0003-2266-2075\\
Departament of Statistics, Mathematical Analysis and Optimization, Faculty of Mathematics and CITMAga, Universidade de Santiago de Compostela (USC), Galicia, Spain;\\

Gonzalo Cao Labora \quad gcaol@mit.edu \quad ORCID: 0000-0002-8426-8391\\
Department of Mathematics, Massachusetts Institute of Technology (MIT), MA, United States of America\\
\end{itemize}

\bibliography{ComplexBorwein}{}
\bibliographystyle{plain}
\end{document}